\documentclass[11pt]{amsart}

\usepackage{graphicx}
\usepackage{amsmath,amsthm}
\usepackage{amssymb}
\usepackage{pinlabel}
\usepackage[left = 3cm, right = 3 cm , top = 3 cm , bottom = 3cm ]{geometry}
\usepackage[utf8]{inputenc}
\usepackage{todonotes}
\usepackage{enumerate}
\usepackage{undertilde}
\usepackage{tikz-cd}
\usepackage{hyperref}
\hypersetup{
    colorlinks=true, 
    linktoc=all,     
    citecolor=black,
    filecolor=black,
    linkcolor=blue,
    urlcolor=black
}
\usepackage{cleveref}

\newtheorem{thm}{Theorem}[section]
\newtheorem{prop}[thm]{Proposition}

\newtheorem{cor}[thm]{Corollary}

\newtheorem{question}[thm]{Question}
\newtheorem{notation}[thm]{Notation}

\theoremstyle{remark}
\newtheorem{remark}[thm]{Remark}

\newtheorem{ex}[thm]{Example}

\theoremstyle{definition}
\newtheorem{definition}[thm]{Definition}

\begin{document}
	
	\title[]{On $C^0$-stability of compact leaves with amenable fundamental group}
	%
	%
	%
\author{Sam Nariman}
\address{Department of Mathematics\\
  Purdue University\\
150 N. University Street\\
West Lafayette, IN 47907-2067\\
}
\email{snariman@purdue.edu}
\author{Mehdi Yazdi}
\address{Department of Mathematics\\
King's College London \\
Strand, London, WC2R 2LS}
\email{mehdi.yazdi@kcl.ac.uk}
	\maketitle
	
	\begin{abstract} In his work on the generalization of the Reeb stability theorem (\cite{MR356087}), Thurston conjectured that if the fundamental group of a compact leaf $L$ in a codimension-one transversely orientable foliation is amenable and if the first cohomology group $H^1(L;\mathbb{R})$ is trivial, then $L$ has a neighborhood foliated as a product. This was later proved as a consequence of Witte-Morris' theorem on the local indicability of amenable left orderable groups and Navas' theorem on the left orderability of the group of germs of orientation-preserving homeomorphisms of the real line at the origin. In this note, we prove that Thurston's conjecture also holds for any foliation that is sufficiently close to the original foliation. Hence, if the fundamental group $\pi_1(L)$ is amenable and $H^1(L;\mathbb{R})=0$, then for every transversely orientable codimension-one foliation $\mathcal{F}$ having $L$ as a leaf, there is a neighborhood of $\mathcal{F}$ in the space of $C^{1,0}$ foliations with Epstein $C^0$ topology consisting entirely of foliations that are locally a product $L \times \mathbb{R}$. 

	\end{abstract}
	
	\section{Introduction} 
A codimension-$m$ \emph{foliation} of an $n$-dimensional manifold $M$ is a decomposition of $M$ into injectively immersed $(n-m)$-dimensional submanifolds such that, locally the submanifolds form a product $\mathbb{R}^{n-m} \times \mathbb{R}^m$ by slices $\mathbb{R}^{n-m} \times \{ \text{point}\}$. The connected components of the submanifolds in this decomposition are called \emph{leaves}. Thus, locally $\mathcal{F}$ is given by an atlas whose transition functions preserve the planes $\mathbb{R}^{n-m}\times \{\text{point}\}$. If the transition functions are $C^k$ with respect to $\mathbb{R}^{n-m}$ coordinates and $C^r$ with respect to $\mathbb{R}^m$ coordinates, then the foliation is of \emph{class $C^{k, r}$}. By a $C^r$ foliation, we mean a foliation of class $C^{r,r}$. 

Just as the dynamics of a function heavily depend on its regularity class, the global properties of a foliation are influenced by its regularity class. For example, Denjoy's theorem (\cite{denjoy1932courbes}) as well as Kopell's Lemma (\cite{kopell2006commuting}) can each be used to give examples of codimension-one foliations that are not topologically conjugate to any $C^2$ foliation. Harrison's results on nonsmoothable diffeomorphisms (\cite{harrison1975unsmoothable, harrison1979unsmoothable}) imply, via the mapping torus construction, that for every non-negative integer $r \geq 0$ and every codimension $m \neq 1,4$, there are codimension-$m$ $C^r$ foliations that are not topologically conjugate to any $C^{r+1}$ foliation.   Cantwell and Conlon (\cite{cantwell1988smoothability}) showed that there are codimension-one $C^0$ foliations that are not topologically conjugate to any $C^1$ foliations. Independently, Tsuboi (\cite{tsuboi1987examples}) showed that for each non-negative integer $r \geq 0$, there are $C^r$ actions of the finitely presented group $\langle a, b | [a, [a,b]]=1 \rangle $ on the interval that are not topologically conjugate to any $C^{r+1}$ action.  See \cite{deroin2007dynamique, calegari2008nonsmoothable, navas2010finitely, bonatti2019hyperbolicity, kim2020diffeomorphism, kim2021direct} for more recent results on the degree of regularity for group actions on $1$-manifolds.

A compact leaf $L$ of a codimension-$m$ foliation $\mathcal{F}$ is \emph{globally stable} if every leaf of $\mathcal{F}$ is diffeomorphic to $L$ and has a neighborhood that is foliated as a product $L \times \mathbb{R}^m$ by leaves $L \times \text{point}$. For codimension-one transversely orientable foliations on compact connected manifolds, the existence of a globally stable leaf $L$ implies that $\mathcal{F}$ is either the product foliation $L \times [0,1]$ or a fibration over the circle.  The stability of compact or proper leaves of foliations, especially in codimension-one, has been extensively studied. See \cite{MR0055692, MR356087, MR461523, plante1983stability, MR682061, fukui1986stability, MR935527, bonatti1990stabilite, bonatti1990diffeomorphismes, tsuboi1994hyperbolic, fukui2000codimension, MR1978886, MR2609248, MR2917081, MR4002563} and references therein. The classical theorem of Reeb stability states that for a codimension-one transversely orientable foliation $\mathcal{F}$ on a compact connected manifold $M$, a compact leaf $L$ with finite $\pi_1(L)$ is globally stable. Thurston (\cite{MR356087}) vastly generalized Reeb's stability theorem, by showing that for a codimension-one transversely orientable $C^1$ foliation $\mathcal{F}$ of a connected manifold, a compact leaf $L$ satisfying $H^1(L;\mathbb{R})=0$ is globally stable. 

Langevin and Rosenberg (\cite{MR461523}) observed that Thurston's proof is invariant under $C^1$-perturbations. Therefore, if $\mathcal{F}'$ is $C^1$-close to the foliation $\mathcal{F}$, then $\mathcal{F}'$  has a compact leaf $L'$ homeomorphic to $L$, and $L'$ is also globally stable. We shall briefly recall in \Cref{Epstein} what an open neighborhood in the $C^r$-topology on the space of foliations defined by Epstein looks like. See \cite{MR0501006} for more detail. Epstein (\cite{MR0501006}) defined this topology for integers $r \geq 1$ or $r = \infty$, but the same definition can be done for $r = 0$. 


While the Reeb stability theorem holds for $C^0$ foliations, Thurston's stability is not true for $C^0$ foliations in this generality. In fact, Thurston gave an example of a codimension-one foliation on $L \times S^1$ with only one compact leaf homeomorphic to $L$, where $L$ is a homology $3$-sphere whose fundamental group is $\langle a,b,c| a^2=b^3=c^7=abc\rangle$. Hence, he suggested finding a characterization of compact leaves satisfying Reeb's stability in the $C^0$ case. He wrote that it would be ``reasonable to conjecture that if $\pi_1(L)$ is amenable, and if $L$ has a nontrivial $C^0$-foliated neighborhood, then $H^1(L;\mathbb{R})\neq 0$". This conjecture is now known to be true as a consequence of Witte-Morris' theorem (\cite{MR2286034}) and Navas' theorem (\cite[Remark 1.1.13]{deroin2014groups} or \cite[Proposition 3]{MR3369358}): Let the \emph{holonomy group of $L$} be the image of $\pi_1(L)$ under the holonomy homomorphism. Then the holonomy group of $L$ is amenable since it is a quotient of the amenable group $\pi_1(L)$. On the other hand, the holonomy group of $L$ is a subgroup of the group $\utilde{\mathrm{Homeo}}_+(\mathbb{R}, 0)$ of germs of orientation-preserving homeomorphisms of the real line at the origin, which is a left orderable group by Navas' theorem. Hence, the holonomy group of such a compact leaf is amenable and left orderable. Therefore, by Witte-Morris' theorem, it has to be locally indicable which implies that $H^1(L;\mathbb{R})\neq 0$.
	
Our main theorem is to prove the stability of compact leaves with an amenable fundamental group under $C^0$-perturbations. Two foliations $\mathcal{F}$ and $\mathcal{F}'$ on a manifold $M$ are \emph{topologically equivalent} if there is a homeomorphism of $M$ that sends the leaves of $\mathcal{F}$ to the leaves of $\mathcal{F}'$. 

\begin{thm}
Let $M$ be a compact connected smooth manifold. Let $\mathcal{F}$ be a transversely oriented codimension-one $C^{1,0}$ foliation on $M$ with a compact leaf $L$ such that $\pi_1(L)$ is amenable and $H^1(L;\mathbb{R}) = 0$. Then any $C^{1,0}$ foliation that is $C^0$ close to $\mathcal{F}$ (i.e. in Epstein $C^0$ topology) is topologically equivalent to $\mathcal{F}$. 
\label{thm:main}
\end{thm}	

Here is an outline of the proof of Theorem \ref{thm:main}. We are interested in the deformation of a foliation $\mathcal{F}$ by a parameter $s$, where $s$ belongs to a pointed topological space $(S,0)$ and the base point $0 \in S$ corresponds to the initial foliation $\mathcal{F}$. For our application, we think of $S$ as the subset $\{ \frac{1}{n} | n \in \mathbb{N}\} \cup \{0\}$ of the real line with the base point $0$. One might expect that a deformation of $\mathcal{F}$ corresponds to a `deformation of the holonomy representation of $\mathcal{F}$'. This is too good to be true, as can be seen by considering a linear foliation of a two-dimensional torus by an irrational slope: In this case, the holonomy representation of the foliation is trivial since all leaves are simply connected, but there are nontrivial linear deformations of the foliation. However, by the foundational work of  Bonatti and Haefliger \cite{bonatti1990deformations}, for $C^r$ foliations where $r \geq 1$, if instead of the (usual) holonomy representation defined on the fundamental group of a leaf we consider the holonomy representation defined on the \emph{\'etale fundamental groupoid} of the foliation, then the germ of a deformation of $\mathcal{F}$ by a parameter $s \in (S, 0)$ corresponds to the germ of a deformation by $s \in (S, 0)$ of the holonomy map defined on the fundamental groupoid. See Section 3 and Theorem \ref{thm:Haefliger--Bonatti-original} for a precise statement. 

A sequence of foliations close to $\mathcal{F}$ defines a deformation of $\mathcal{F}$ by a parameter $s \in S = \{ \frac{1}{n} | n \in \mathbb{N}\} \cup \{0\}$. By Bonatti and Hafliger's work, this deformation of $\mathcal{F}$ gives a deformation by $s \in S$ of the holonomy representation $\mathrm{H}$ of the fundamental groupoid of $\mathcal{F}$. Now since the foliation $\mathcal{F}$ is locally a product, its holonomy representation $\mathrm{H}$ is easy to describe. As a result, a deformation of the holonomy representation of the fundamental groupoid of $\mathcal{F}$ gives rise to a homomorphism 
\[ \phi \colon \pi_1(L) \rightarrow  \utilde{\mathrm{Homeo}}_+^{(S,0)}(\mathbb{R}, 0),\]
where $\utilde{\mathrm{Homeo}}_+^{(S,0)}(\mathbb{R}, 0)$ is a certain subgroup of the group of germs of local homeomorphisms of $\mathbb{R} \times S$ at the point $(0,0) \in \mathbb{R} \times S$. See Notation \ref{notation: group of germs}. In Proposition \ref{prop: left orderable}, we generalise Navas's theorem to show that $\utilde{\mathrm{Homeo}}_+^{(S,0)}(\mathbb{R}, 0)$ is left-orderable. Since $\pi_1(L)$ is amenable, it follows from Witte-Morris's theorem that $\phi$ should be trivial.  By Theorem \ref{thm:Haefliger--Bonatti} (which proves the injectivity part of Bonatti--Haefliger's theorem in the $C^0$ case), the map 
\[ \{ \text{germs of deformations of }\mathcal{F} \}/\sim \hspace{6mm} \longrightarrow \hspace{6mm} \{ \text{germs of deformations of }\mathrm{H} \}/\sim' \] 
is injective, where $\sim$ and $\sim'$ are the relevant equivalence relations on the two sides. Since $\phi$ is trivial, it follows that the deformation of $\mathcal{F}$ should be trivial as well and the theorem follows. 


\subsection*{Acknowledgment}SN was partially supported by a grant from the Simons Foundation (41000919, SN) and NSF CAREER Grant DMS-2239106.	MY is supported by a UKRI Postdoctoral Research Fellowship.


\section{Left-invariant orders on groups}	
In this section we collect some background on left-orderable groups; a good reference is Clay and Rolfsen (\cite{clay2016ordered}). We then prove that a certain group of germs of homeomorphisms is left-orderable, generalizing Navas's theorem. 

A group $G$ is \emph{left-orderable} if there is a total order $\prec$ on the elements of $G$ that is invariant under left multiplication by $G$; i.e. for every $f, g, h \in G$ 
\[ g \prec h \implies fg \prec fh.   \] 
Note that a left-orderable group is torsion-free, since if $e \in G$ is the identity element and $e \neq g \in G$ is such that $g^n=e$ and $g \succ e$, then 
\[ g^n \succ g^{n-1} \succ \cdots \succ g \succ e \]
which is a contradiction. Similarly, $g \prec e$ leads to a contradiction, implying that $G$ is torsion-free. A subgroup of a left-orderable group is again a left-orderable group, which is easily seen by restricting the order. 

\begin{ex}
	An important example of a left-orderable group is the group $\mathrm{Homeo}^+(\mathbb{R})$ of orientation-preserving homeomorphisms of $\mathbb{R}$. To see this, let $a_1, a_2, a_3, \cdots$ be a dense sequence in $\mathbb{R}$. For every two distinct elements $f, g  \in \mathrm{Homeo}^+(\mathbb{R})$, let $n = n(f,g)$ be the smallest number such that $f(a_n) \neq g(a_n)$, and define an ordering $ \prec $ as
	\[ f \prec g \iff f(a_n)<g(a_n). \]
	Since the $\{ a_i\}$ are dense in $\mathbb{R}$, for every $f \neq g$ such $n$ exists. Moreover, since elements of $\mathrm{Homeo}^+(\mathbb{R})$ are strictly increasing, the above ordering is invariant under left multiplication. 
\end{ex}
The following theorem shows that $\mathrm{Homeo}^+(\mathbb{R})$ is a fairly general example of a left-orderable group. See Deroin, Navas, and Rivas \cite{deroin2014groups} and Navas \cite{navas2010dynamics} for further interplay between dynamics and left-orders. 

\begin{thm}[Dynamic realization]
	Let $G$ be a countable group. Then $G$ is left-orderable if and only if there is an injective homomorphism $h \colon G \rightarrow \mathrm{Homeo}^+(\mathbb{R})$; i.e. $G$ acts faithfully on the real line by orientation-preserving homeomorphisms. 
	\label{thm: dynamic realization}
\end{thm}

\begin{proof}
	We bring the proof given in \cite[Theorem 2.2.19]{navas2011groups} since we need to slightly modify it later for our application. Suppose that $G$ admits a left-invariant order $\prec$. First, define an embedding $t \colon G \rightarrow \mathbb{R}$ as follows: Choose a numbering $g_0, g_1, g_2, \cdots$ for the elements of $G$. Let $t(g_0)=0 \in \mathbb{R}$. Assume that $g_0, g_1, \cdots, g_i$ are defined and define $g_{i+1}$ inductively as follows. If $g_{i+1}$ is larger (in the order $\prec$) than all of the previously defined $g_1, g_2, \cdots, g_i$ then set
	\[ t(g_{i+1}) = \max \{ t(g_0), \cdots, t(g_i) \} +1.\]  
	If $g_{i+1}$ is smaller than all of $g_1, g_2, \cdots, g_i$ then define 
	\[ t(g_{i+1}) = \min \{ t(g_0), \cdots, t(g_i) \} +1.\] 
	If there are $0 \leq m, n \leq i$ such that $g_n \prec g_{i+1} \prec g_{m}$ and no other element of $\{ g_0, g_1, \cdots, g_i \}$ lies strictly between $g_n$ and $g_m$ then define 
	\[ t(g_{i+1}) = \frac{t(g_n)+t(g_m)}{2}. \]
	This defines an embedding $t \colon G \rightarrow \mathbb{R}$. Note that the image of $t$ is unbounded, since if $g \succ e$ then $t(g^n) \rightarrow +\infty$ as $n \rightarrow +\infty$. 
	Then $G$ naturally acts on $t(G)$ by setting $g \cdot (t(g_i)): = t(gg_i)$ for every $g \in G$, and this action continuously extends to the closure $\overline{t(G)}$ of $t(G)$ in $\mathbb{R}$. The complement of the closure of $t(G)$ in $\mathbb{R}$ is an open subset of $\mathbb{R}$, and so is a countable union of open intervals. Extend this action of $G$ on $\overline{t(G)}$ to an action of $G$ on $\mathbb{R}$ such that each $g \in G$ acts by an affine map when restricted to each of the intervals of $\mathbb{R} \setminus \overline{t(G)}$.
\end{proof}

\begin{remark}
	We show that in Theorem \ref{thm: dynamic realization}, one can assume that the following conditions are also satisfied:
	
	\begin{enumerate}
		\item[a)] each element in the image of $h$ fixes the point $0 \in \mathbb{R}$; and
		\item[b)] each element in the image of $h$ has non-trivial germ at $0$. 
	\end{enumerate}
	These two conditions will be used later in the proof of Corollary \ref{thm:nontrivial homomorphism}. 
	
	Let $h \colon G \rightarrow \mathrm{Homeo}^+(\mathbb{R})$ be the constructed action in the proof of Theorem \ref{thm: dynamic realization}. Identify $\mathbb{R}$ with the interval $(-\infty, 0) \subset \mathbb{R}$ to obtain a new action $h' \colon G \rightarrow \mathrm{Homeo}^+(\mathbb{R})$. Then $h'$ has the desired properties. Clearly, every element in the image of $h'$ fixes the point $0$. Moreover, each element in the image of $h'$ has a non-trivial germ at $0$. This is because for every element $g \succ e$ we have $h'(g)^n(t(g_0)) \rightarrow 0$ as $n \rightarrow + \infty$.
	\label{rem: dynamic realization}
\end{remark}

If $X$ is a subset of a group $G$, denote by $S(X)$ the semigroup generated by elements of $X$. If $X = \{ x_1, \cdots, x_n\}$ is a finite set, we write $S(X)$ as $S(x_1, \cdots, x_n)$. For proof of the following criterion for left-orderability see Clay and Rolfsen \cite[Theorem 1.48]{clay2016ordered} or Deroin, Navas, and Rivas \cite[Section 1.1.2]{deroin2014groups}.

\begin{thm}
	A group $G$ is left-orderable if and only if for every finite subset $\{ x_1, \cdots, x_n\}$ of $G$ which does not contain the identity, there exist $\epsilon_i \in \{ -1, 1\}$ such that $S(x_1^{\epsilon_1}, \cdots, x_n^{\epsilon_n})$ does not contain the identity element of $G$.
	\label{thm: criterion for left-orderability}
\end{thm}

Intuitively, the above criterion says that if there is no obstruction for defining a left-order on each \emph{finitely generated} subgroup of the group $G$, then $G$ is left-orderable. 

\begin{notation}
	Let $S$ be a topological space and $0 \in S$ be a base point. Let $\utilde{\mathrm{Homeo}}_+^{(S,0)}(\mathbb{R}, 0)$ be the group of germs of local homeomorphisms of $\mathbb{R} \times S$ of the form $(h_s(x), s)$ at the point $(0 \times 0)$ where each $h_s(x)$ is a local orientation-preserving homeomorphism of $\mathbb{R}$. In particular, all elements of $\utilde{\mathrm{Homeo}}_+^{(S,0)}(\mathbb{R}, 0)$ fix the point $(0, 0)$.
	\label{notation: group of germs}
\end{notation}

The following is an analog of Navas's theorem (\cite[Remark 1.1.13]{deroin2014groups} and \cite[Proposition 3]{MR3369358}). The case of $S = \{ 0 \}$ is the statement of Navas's theorem. The proof is also similar. 

\begin{prop}
	Let $(S,0)$ be a pointed metrizable topological space. The group $\utilde{\mathrm{Homeo}}_+^{(S,0)}(\mathbb{R}, 0)$ is left-orderable.
	\label{prop: left orderable}
\end{prop}

\begin{proof}
	Let $G = \utilde{\mathrm{Homeo}}_+^{(S,0)}(\mathbb{R}, 0)$.  By Theorem \ref{thm: criterion for left-orderability}, it is enough to show that for every non-identity elements $f_1, \cdots, f_k$ in $G$, there are $\epsilon_1, \cdots, \epsilon_k$ in $\{-1 , +1\}$ such that the semigroup generated by $f_i^{\epsilon_i}$ does not contain the identity element. 
	
	Note that the identity element, denoted by $e$, is the germ of the identity homeomorphism of $\mathbb{R} \times S$ at $(0 \times 0)$. Hence we have the following criterion for detecting non-identity elements: $f \in G$ is not equal to the identity element $e$ if and only if there exists a sequence $(p_j, q_j) \in \mathbb{R} \times S$ of points converging to $(0, 0)$ as $j \rightarrow \infty$ such that $f(p_j, q_j) \neq (p_j, q_j)$. Here we use the fact that $S$ is metrizable to deduce that convergence and sequential convergence are equivalent.
	
	For each $1 \leq i \leq k$, pick a representative homeomorphism for $f_i$ and denote it by $f_i$ again by abuse of notation. We first define $\epsilon_1, \cdots, \epsilon_k$, starting with $\epsilon_1$. Since $f_1 \neq e$, there is a sequence $(x_r, s_r) \rightarrow (0, 0)$ such that for every $r$ 
	\[ f_1(x_r, s_r) \neq (x_r, s_r). \]
	Let $\pi_1 \colon \mathbb{R} \times S \rightarrow \mathbb{R}$ be the projection map onto the first factor. After passing to a subsequence of $(x_r, s_r)$ we may assume that either 
	\begin{enumerate}
		\item 	for every $r$ we have $ \pi_1 \circ f_1 (x_r, s_r) > x_r$; or
		
		\item	for every $r$ we have $ \pi_1 \circ f_1 (x_r, s_r) < x_r$.
	\end{enumerate}

	In case (1) define $\epsilon_1=1$, and in case (2) set $\epsilon_1=-1$.
	
	Now consider $f_2$. At least one of the following holds:
	\begin{enumerate}
		\item[i)] for infinitely many $r$ we have $\pi_1 \circ f_2 (x_r, s_r) > x_r$; or 
		\item[ii)] for infinitely many $r$ we have $\pi_1 \circ f_2 (x_r, s_r) < x_r$; or 
		\item[iii)] for all but finitely many $r$ we have $\pi_1 \circ f_2 (x_r, s_r) = x_r$.
	\end{enumerate}
	In case i) define $\epsilon_2 =1 $, in case ii) set $\epsilon_2 = -1$, and in case iii) we declare $\epsilon_2$ undefined for the moment. In cases i) and ii), after passing to a subsequence of $(x_r, s_r)$ we can assume that $\pi_1 \circ f_2^{\epsilon_2} (x_r, s_r) > x_r$ for all $r$. Repeat this procedure for $f_3, \cdots, f_k$ to define a subset of $\epsilon_1, \cdots, \epsilon_k$. After reordering, assume that $\epsilon_1, \cdots, \epsilon_\ell$ are defined for some $1 \leq \ell \leq k$. Hence for every $1 \leq i \leq \ell$ and every $r$ we have $\pi_1 \circ f_i^{\epsilon_i} (x_r, s_r) > x_r$, and for every $i> \ell$ and every $r$ we have $\pi_1 \circ f_i (x_r, s_r) = x_r$. 
	
	Now start with $f_{\ell +1} \neq e$ and choose a new sequence $(x'_r, s'_r) $ converging to $(0, 0)$ as $r \rightarrow \infty$ such that $f_{\ell +1}(x'_r, s'_r) \neq (x'_r, s'_r)$ for every $r$. We then repeat the previous procedure using the sign of $f_{\ell +1}(x'_r, s'_r) - (x'_r, s'_r)$ to define a non-empty subset of $f_{\ell +1} , \cdots, f_k$. Repeating this procedure we define all of $\epsilon_1, \cdots, \epsilon_k$ using a finite number, say $N$, of sequences $\{(x_r, s_r)\}_{r=1}^{\infty}, \{(x'_r, s'_r)\}_{r=1}^{\infty}, \cdots$ with each sequence converging to $(0,0)$.
	Put the sequences $\{(x_r, s_r)\}_{r=1}^{\infty}, \{(x'_r, s'_r)\}_{r=1}^{\infty}, \cdots$ in successive rows of a table with $N$ rows and infinitely many columns, and define the sequence $(p_j,q_j)$ by following successive finite diagonals of this table. 
	
	Let $w \in S(f_1^{\epsilon_1}, \cdots, f_k^{\epsilon_k})$ be a non-empty word. We show that there are infinitely many $j$ such that $w(p_j, q_j) \neq (p_j,q_j)$. Consider two cases:
	\begin{enumerate}
		\item[a)] At least one of $f_1, \cdots, f_{\ell}$ appears in $w$. Recall that $\ell$ was the index such that $f_1, \cdots, f_\ell$ were defined using the first sequence $(x_r, s_r)$. By construction, for each $1 \leq m \leq \ell $ we have $\pi_1 \circ f_m^{\epsilon_m} (x_r, s_r) > x_r$ and for every $m> \ell$ we have $\pi_1 \circ f_m^{\epsilon_m} (x_r, s_r) = x_r$. Therefore we have  $\pi_1 \circ w(x_r, s_r) > x_r$ for every $r$. Since $(x_r, s_r)$ is an infinite subsequence of $(p_j, q_j)$ we are done. 
		
		\item[b)] None of $f_1, \cdots, f_{\ell}$ appear in $w$. In this case, we have a word in $f_{\ell+1}, \cdots, f_k$ and we can use the next sequence $(x'_r, s'_r)$ to argue similarly. 
	\end{enumerate}
\end{proof}

\begin{cor}
	Let $G$ be a countable group. The following statements are equivalent: 
	\begin{enumerate}
		\item $G$ has a non-trivial homomorphism into $\mathrm{Homeo}_+(\mathbb{R})$;
		\item $G$ has a non-trivial homomorphism into $\utilde{\mathrm{Homeo}}_+(\mathbb{R}, 0)$;
		\item $G$ has a non-trivial homomorphism into $\utilde{\mathrm{Homeo}}_+^{(S,0)}(\mathbb{R}, 0)$, where $S = \{ 0\} \cup \{\frac{1}{n} | n \in \mathbb{N} \} \subset \mathbb{R}$;
		\item $G$ has a non-trivial homomorphism into $\utilde{\mathrm{Homeo}}_+^{(S,0)}(\mathbb{R}, 0)$ for some pointed metrizable topological space $(S,0)$.
	\end{enumerate} 
	\label{thm:nontrivial homomorphism}
\end{cor}

\begin{proof}
	The implications $(2) \implies (3) \implies (4)$ are immediate. By Navas' theorem, we have $(2) \implies (1)$. To see this, let 
	\[ \phi \colon G \rightarrow \utilde{\mathrm{Homeo}}_+(\mathbb{R}, 0) \] 
	be a homomorphism with a non-trivial image. Then the image $\phi(G)$ is countable and left orderable. Therefore, $\phi(G)$ is a subgroup of $\mathrm{Homeo}_+(\mathbb{R})$ and so there is an injective homomorphism $i \colon \phi(G) \hookrightarrow \mathrm{Homeo}_+(\mathbb{R})$. The composition $i \circ \phi$ then gives a non-trivial homomorphism of $G$ into $\mathrm{Homeo}_+(\mathbb{R})$. 
	
	The same argument shows that $(4) \implies (1)$, where we use Proposition \ref{prop: left orderable} instead of Navas' theorem.
	
	The implication $(1) \implies (2)$ is also well-known: Let 
	\[ \phi \colon G \rightarrow \mathrm{Homeo}_+(\mathbb{R}) \]
	be a homomorphism with a non-trivial image. The image $\phi(G)$ is countable and left orderable. Therefore, by the dynamic realization construction (Theorem \ref{thm: dynamic realization}), there is an injective homomorphism 
	\[ \psi \colon \phi(G) \rightarrow \mathrm{Homeo}_+(\mathbb{R}) \]
	such that 
	\begin{enumerate}
		\item[a)] each element in the image of $\psi$ fixes the point $0 \in \mathbb{R}$; and
		\item[b)] each element in the image of $\psi$ has non-trivial germ at $0$. 
	\end{enumerate}
	See Remark \ref{rem: dynamic realization}. It follows that $\psi$ induces a non-trivial homomorphism $\overline{\psi} \colon \phi(G) \rightarrow \utilde{\mathrm{Homeo}}_+(\mathbb{R}, 0)$. Then $\overline{\psi}\circ \phi$ is a non-trivial homomorphism from $G$ into $\utilde{\mathrm{Homeo}}_+(\mathbb{R}, 0)$. 
	
	This completes the equivalence of items (1)--(4).
\end{proof}

\section{Deformations of foliations and their holonomies}

We begin by defining the notion of an \'{e}tale groupoid following the exposition of Haefliger \cite{haefliger2001groupoids}. For further examples see Bridson and Haefliger \cite[Chapter III.$\mathcal{G}$]{bridson2013metric}.

\subsection{\'{E}tale groupoids}

A \emph{groupoid} $(\mathcal{G}, T)$ is a small category with a set of objects $T$ and morphisms $\mathcal{G}$ such that all elements of $\mathcal{G}$ are invertible. There are two projections: the \emph{source projection} $s \colon \mathcal{G} \rightarrow T$ and the \emph{target projection} $t \colon \mathcal{G} \rightarrow T$. The composition $gg'$ of two elements $g$ and $g'$ of $\mathcal{G}$ is defined if $s(g)= t(g')$. 

A \emph{topological groupoid} $(\mathcal{G}, T)$ is a groupoid such that $\mathcal{G}$ and $T$ are topological spaces, with the following properties: 
\begin{itemize}
	\item the composition and taking inverse are continuous; and
	\item the inclusion $x \rightarrow 1_{\{x\}}$ is a homeomorphism $T \rightarrow \mathcal{G}$ onto its image, where $1_{x}$ is the identity morphism at $x$.
\end{itemize}

An \emph{\'{e}tale groupoid} is a topological groupoid $(\mathcal{G}, T)$ such that the source and target projections are \'{e}tale maps; i.e. are locally homeomorphisms. 

Given an \'{e}tale groupoid $(\mathcal{G}, T)$ and an open cover $\mathcal{U}=\{ U_i\}_{i \in I}$ of $T$, the \emph{localization of $(\mathcal{G}, T)$ over $\mathcal{U}$} is an \'{e}tale groupoid $(\mathcal{G}_{\mathcal{U}}, T_{\mathcal{U}})$ defined as follows: 
\begin{itemize}
	\item here $T_{\mathcal{U}}$ is the disjoint union of open sets $U_i$;
	\item the elements of $\mathcal{G}_{\mathcal{U}}$ are the triples $(i,g,j)$ with $s(g) \in U_j$ and $t(g) \in U_i$;
	\item the source and target projection maps send $(i, g, j)$ to $(j,s(g))$ and $(i, t(g))$ respectively; and
	\item the composition $(i,g,j)(j,h,k)$ is defined as $(i, gh, k)$.
\end{itemize}

  Let $\Gamma$ and $\Gamma'$ be two \'{e}tale groupoids whose spaces of objects (units as the space of morphisms) are $T$ and $T'$ respectively.  A \emph{homomorphism} from $\Gamma$ to $\Gamma'$ is a continuous functor that in particular induces a local homeomorphism between morphism spaces. But the notion of a homomorphism is restrictive since we want to work with \'{e}tale groupoids up to equivalence. 
  
  \begin{definition}[Equivalence of \'{e}tale groupoids]
  	Two \'{e}tale groupoids $(\mathcal{G}, T)$ and $(\mathcal{G}',T')$ are \emph{equivalent} if there is an open cover $\mathcal{U}$ of $T$ and an open cover $\mathcal{U}'$ of $T'$ such that the localizations $(\mathcal{G}_\mathcal{U}, T_\mathcal{U})$ and $(\mathcal{G}'_{\mathcal{U}'}, T'_{\mathcal{U}'})$ are isomorphic.
  	\label{def:equivalent}
  \end{definition}
  
  In particular, as we shall see, we want to consider the fundamental groupoid of foliations whose definition depends on the choice of a complete transversal but different choices give equivalent groupoids. So we consider the following notion of \emph{morphisms} between \'{e}tale groupoids. 
 \begin{definition} [Morphism between \'{e}tale groupoids]  Let $\Gamma$ and $\Gamma'$ be two \'{e}tale groupoids whose space of units  are $T$ and $T'$ respectively. Let $\mathcal{U}$ and $\mathcal{V}$ be two open covers of $T$, and let $\phi_{\mathcal{U}}\colon \Gamma_{\mathcal{U}}\to \Gamma'$ and $\phi_{\mathcal{V}}\colon \Gamma_{\mathcal{V}}\to \Gamma'$ be two homomorphisms. Denote by $\mathcal{U} \coprod \mathcal{V}$ the open cover of $T$ that is the disjoint union of the open covers $\mathcal{U}$ and $\mathcal{V}$. We say $\phi_{\mathcal{U}}$ and $\phi_{\mathcal{V}}$ represent equivalent \emph{morphism} $\phi\colon \Gamma\to \Gamma'$ if there exists a homomorphism $\phi_{\mathcal{U}\coprod \mathcal{V}}\colon \Gamma_{\mathcal{U}\coprod \mathcal{V}}\to \Gamma'$ extending $\phi_{\mathcal{U}}$ and $\phi_{\mathcal{V}}$.
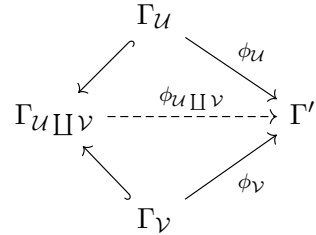
 \[
 \begin{tikzcd}[column sep=tiny]
& \Gamma_{\mathcal{U}} \ar[dr, "\phi_{\mathcal{U}}"] \arrow[dl, hook]
&[1cm]
\\
   \Gamma_{\mathcal{U}\coprod \mathcal{V}} \arrow[rr, dashrightarrow, "\phi_{\mathcal{U}\coprod \mathcal{V}}"]
    &
      &\Gamma' 
 \\
&\Gamma_{\mathcal{V}} \ar[ur, "\phi_{\mathcal{V}}"']\arrow[ul, hook]
&
\end{tikzcd}
\]  

\label{def: equivalent morphisms}
 \end{definition}

\subsection{Fundamental groupoid of a foliation}

 Bonatti and Haefliger (\cite{bonatti1990deformations}) developed a theory for studying germs of deformations of a foliation and germs of deformations of its holonomy. See \cite{Bonatti1993} for a comprehensive exposition and various applications. Let $\mathcal{F}$ be a $k$-dimensional $C^{1,0}$ foliation of a compact $n$-dimensional manifold $M$ and $T\mathcal{F}$ be the tangent bundle to $\mathcal{F}$. A \emph{transversal to the foliation $\mathcal{F}$} is a (possibly disconnected) manifold $T$ of dimension $n-k$ and an immersion $\tau \colon T \rightarrow M$ such that at each point $x \in T$ we have 
 \[ T_{\tau(x)}(M) = \tau_*(T_x(T)) \oplus T_{\tau(x)}(\mathcal{F}). \]
 By abuse of notation, we refer to the transversal by $T$. A transversal $T$ is \emph{complete} if it intersects all leaves of $\mathcal{F}$.

\begin{definition}[Fundamental groupoid]
	Let $\mathcal{F}$ be a foliation of a manifold $M$, and $\tau \colon T \rightarrow M$ be a complete transversal. The \emph{fundamental groupoid of $\mathcal{F}$ for the transversal $T$}, denoted by $\Pi_{\mathcal{F}}(T)$, is an \'{e}tale groupoid defined as follows: 
	\begin{itemize}
		\item As a groupoid, its objects are the points in $T$. A morphism from a point $p \in T$ to a point $q \in T$ is the homotopy class of a path $\gamma$ in a leaf of $\mathcal{F}$ starting at $\tau(p)$ and ending at $\tau(q)$, where all the paths during the homotopy lie in the same leaf of $\mathcal{F}$.
		
		\item An open basis for the topology on the set of morphisms is obtained as follows: Let $h$ be a homeomorphism from an open subset $U$ of $T$ to an open subset $V$ of $T$ such that there exists a continuous map $C \colon U \times [0,1] \rightarrow M$ such that, for each $u \in U$, $C_u \colon t \rightarrow C(u,t)$ is a path contained in a leaf of $\mathcal{F}$, starting at $\tau(u)$ and terminating at $\tau(h(u))$. Then the collection of homotopy classes in the leaves of $\mathcal{F}$ of the paths $C_u$, $u \in U$ is such an open set.
	\end{itemize}  
\end{definition}

Note the fundamental groupoid of a foliation $\mathcal{F}$ depends on the choice of a complete transversal $T$. However for two complete transversals $T$ and $T'$ the fundamental groupoids $\Pi_\mathcal{F}(T)$ and $\Pi_\mathcal{F}(T')$ are equivalent in the sense of Definition \ref{def:equivalent}.

There is a natural map that associates to each path $\gamma$ tangent to $\mathcal{F}$ its holonomy germ, and induces a homomorphism from $\Pi_\mathcal{F}(T)$ to the groupoid $\utilde{\mathrm{Homeo}}(T)$ of germs of local homeomorphisms of $T$ 
\begin{eqnarray*}
	\mathrm{H} \colon \Pi_\mathcal{F}(T) \rightarrow \utilde{\mathrm{Homeo}}(T).
\end{eqnarray*} 

The {\' e}tale groupoid $ \utilde{\mathrm{Homeo}}(T)$ has $T$ as its space of objects, and morphisms between two points $x$ and $y$ in $T$ are given by the set of germs of homeomorphisms sending $x$ to $y$. The union of these sets as $x$ and $y$ vary is the morphism space that has the so-called sheaf topology. An open neighborhood of a germ $f$ sending $x$ to $y$ in the morphism space is described as follows. Let $F$ be a local homeomorphism of $T$ from an open set $U$ containing $x$ to an open set $V$ containing $y$ such that its germ at $x$ is $f$. The germs of $F$ at points in $U$ give an open neighborhood of $f$ in the morphism space of $ \utilde{\mathrm{Homeo}}(T)$.

\subsection{Epstein topology on the space of foliations}\label{Epstein}

We want to define the holonomy map for deformations of foliations as an {\' e}tale map out of the fundamental groupoid. To do so, we shall first recall the Epstein topology on the space of foliations.

 For a codimension-$m$ foliation $\mathcal{F}$ of an $n$-dimensional manifold $M$, we choose a {\it neighborhood scheme} $S=(I, \{\phi_i\}, \{K_i\})$ where for each $i$ in the index set $I$, the map $\phi_i\colon U_i\to \mathbb{R}^n$ is a $C^r$-diffeomorphism from an open subset $U_i$ of $M$ to an open set in $\mathbb{R}^n$, such that the image of each leaf in $U_i$ is parallel to $\mathbb{R}^{n-m}\times\{\text{point}\}$, and $K_i\subset U_i$ gives a locally finite family of compact sets in $M$ such that $\phi_i(K_i)$ is a closed $n$-cube with sides parallel to the axes of $\mathbb{R}^n$, and the interiors $\{\text{Int}(K_i)\}$ covers $M$. Each pair $(\phi\colon U\to \mathbb{R}^n, K)$ with the same properties is called a {\it distinguished chart} for $\mathcal{F}$. Given a neighborhood scheme $S$, an index $i\in I$, and a positive real number $\delta$, we define the neighborhood $N(i,\delta)$ of $\mathcal{F}$ in $C^r$-topology to be the set of those foliations $\mathcal{F}'$ for which there is a distinguished chart $(\phi', K)$ for $\mathcal{F}'$ such that $K_i\subset K$ and $|\phi'\circ\phi_i^{-1}-\text{Id}|_r\leq \delta$ on $\phi_i(K_i)$ where $|\cdot|_r$ is the $C^r$-norm on functions on $\mathbb{R}^n$. Then a basis in Epstein's $C^r$ topology on the space of foliations is given by finite intersections of such neighborhoods.

Let $\mathrm{Fol}_m^r(M)$ be the space of codimension-$m$ $C^r$ foliations on $M$ endowed with Epstein $C^r$ topology. Epstein (\cite{MR0501006}) proved that this topology on the space of foliations satisfies the following two axioms: The first axiom states that the group of $C^r$ homeomorphisms of $M$ acts continuously on $\mathrm{Fol}_m^r(M)$. Intuitively, the second axiom states that every foliation $\mathcal{F}'$ sufficiently close to $\mathcal{F}$ has holonomy defined and sufficiently close to that of $\mathcal{F}$. More precisely, let $D(r)$ be the open ball of radius $r$ in $\mathbb{R}^m$ centered at the origin. Let $h \colon I \times D(1) \rightarrow M$ be a $C^r$ map such that $h|\{t\} \times D(1)$ is an embedding transverse to $\mathcal{F}$ for each $t \in I$ such that for each $x \in D(1)$, $h(I \times \{x\})$ lies in a single leaf of $\mathcal{F}$. We require that if $\mathcal{F}'$ is sufficiently close to $\mathcal{F}$, there is a $C^r$ map $k \colon I \times D(\frac{1}{2}) \rightarrow M$ such that 

\begin{enumerate}
	\item[i)] $k| 0 \times D(\frac{1}{2}) = h | 0 \times D(\frac{1}{2})$;
	\item[ii)] for each $x \in D(\frac{1}{2})$, $k|I \times \{x\}$ lies on a single leaf of $\mathcal{F}'$;
	\item[iii)] $k|t \times D(\frac{1}{2}) \subset h|t \times D(\frac{3}{4})$, and $k|t \times D(\frac{1}{2})$ is an embedding for each $t$;
	\item[iv)] for each $t$, $\mathcal{F}'$ is transverse to $h|t \times D(\frac{3}{4})$;
	\item[v)] $k$ is $C^r$-near to $h| I \times D(\frac{1}{2})$. 
\end{enumerate}

We only need the second axiom in this article. Schweitzer (\cite[Lemma 1.1 and Proposition 4.1]{MR935527}) gave a simplified proof of a version of the second axiom for the $C^1$ case; it is straightforward to see that his proof works for the $C^0$ case as well.

\subsection{Deformation of a foliation}
Here we recollect part of the main theorem of Bonatti- Haefliger  (\cite{bonatti1990deformations}) for the deformations of $C^r$-foliations for $r>0$ that also holds for $C^0$-case. 
Let $(S, 0)$ be a locally compact topological space with a base point $0 \in S$. We think of $S$ as the parameter space for deforming either the foliation or its holonomy; here the initial foliation $\mathcal{F}$ corresponds to the base point $0 \in S$. Two useful examples to keep in mind are those of $S$ being $\{ 0 \} \cup \{ \frac{1}{n} | n \in \mathbb{N}\} \subset [0,1]$ with the subspace topology and the base point $0$, or $S$ being the interval $(-1,1)$ with the base point $ 0$. 

We want to consider the {\it germ of a foliation $\mathcal{F}^S$ of $M\times S$ around $M\times \{0\}$}. We shall first recall some definitions to make sense of foliations on $M\times S$ when $S$ is only a topological space.  Let $\mathrm{Homeo}^{S}(T)$ be the pseudogroup of local homeomorphisms of $T \times S$ of the form $(h_s(x), s)$ where $h_s$ is a local homeomorphisms of $T$ continuously varying with $s$. Denote by $\utilde{\mathrm{Homeo}}^{(S,0)}(T)$ (respectively $\utilde{\mathrm{Homeo}}^{S}(T)$) the groupoid of germs of elements of $\mathrm{Homeo}^{S}(T)$ at points of $T \times \{ 0\}$ (respectively $T \times S$).

There is a natural projection 
\[ \pi \colon \utilde{\mathrm{Homeo}}^{(S, 0)}(T) \rightarrow \utilde{\mathrm{Homeo}}(T). \]

A \emph{foliation $\mathcal{F}^S$ parametrized by $S$} on an open set $U$ in $M\times S$ is given by the following cocycle data:
\begin{itemize}
\item Let $\{U_i\}_{i\in I}$ be an open cover of $U$ in $M\times S$. And let $U_i^s$ be the intersection $U_i\cap (M\times\{s\})$. 
\item For each $i$, we have a map $f_i$ from $U_i$ to an open set in $\mathbb{R}^n\times S$, of the form $f_i(x,s)=(f_i^s(x), s)$ where $f_i^s$ is a submersion from $U_i^s$ into $\mathbb{R}^n$ which varies continuously with $s$. 
\item For each $i$ and $j$, we have a continuous map $g_{ij}: U_i\cap U_j\to  \utilde{\mathrm{Homeo}}^{S}(\mathbb{R}^n)$ such that for all $(x,s)\in U_i\cap U_j$, the maps $g_{ij}(x,s)\circ f_j$  and $f_i$ have the same germ at $(x,s)$. 
\end{itemize}

Two such cocycles $(\mathcal{U}, f_i, g_{ij}), (\mathcal{V}, h_i, k_{ij})$ define the same foliation $\mathcal{F}^S$ on $U$ if one can simultaneously extend these two cocycles to a cocycle for the open covering $\mathcal{U}\coprod \mathcal{V}$ of $U$. Note that the foliation $\mathcal{F}^S$ defined in this way on $U$ gives a foliation $\mathcal{F}^s$ on $U^s=U\cap (M\times\{s\})$ which varies continuously with $s$.
\begin{definition}[Germ of deformation of a foliation]
	Let $\mathcal{F}$ be a foliation of a compact manifold $M$. A \emph{local deformation parametrized by $(S, 0)$ of $\mathcal{F}$} is a foliation $\mathcal{F}^S$ parametrized by $S$ of an open neighborhood $U$ of $M\times \{0\}$ in $M\times S$ such that $\mathcal{F}^0$ is the foliation $\mathcal{F}$ on $M\times \{0\}$. Two such local deformations of $\mathcal{F}$ have the same \emph{germ} if they agree on a neighborhood of $M \times \{ 0\}$ in $M \times S$. A \emph{germ of deformation parametrized by $(S, 0)$ of $\mathcal{F}$} is the germ of a local deformation parametrized by $(S, 0)$ of $\mathcal{F}$. Two germs of foliations $\mathcal{F}_1^S$ and $\mathcal{F}_2^S$ are defined to be \emph{equivalent} if there exist open neighborhoods $U_1$ and $U_2$ of $M\times \{0\}$ in $M\times S$, and a homeomorphism $\phi\colon U_1\to U_2$ that lies in $\mathrm{Homeo}^S(M)$ such that the restriction of $\phi$ to $M\times \{0\}$ is the identity and that $\phi$ maps $\mathcal{F}_1^S|_{U_1}$ to $\mathcal{F}_2^S|_{U_2}$.
\end{definition}

 Bonatti and Haefliger gave an algebraic analog of a germ of deformation of foliation by defining the notion of deformation of holonomy. 
\begin{definition}[Germ of deformation of holonomy for the transversal $T$]
	Let $\mathcal{F}$ be a foliation and $T$ be a complete transversal for $\mathcal{F}$. We define a \emph{germ of deformation (parametrized by $(S, 0)$) of the holonomy of  $\mathcal{F}$ for the transversal $T$} to be a homomorphism $\mathrm{H}^S$ from $\Pi_{\mathcal{F}}(T)$ into $ \utilde{\mathrm{Homeo}}^{(S, 0)}(T)$ that makes the following diagram commutative:
	\begin{eqnarray*}
		\begin{tikzcd}
			&  \utilde{\mathrm{Homeo}}^{(S, 0)}(T) \arrow[d, "\pi"] \\
			\Pi_{\mathcal{F}}(T) \arrow[ru, "\mathrm{H}^S"] \arrow[r, "\mathrm{H}"] & \utilde{\mathrm{Homeo}}(T)
		\end{tikzcd}
	\end{eqnarray*}
\end{definition}

\begin{definition}[Germ of deformation of holonomy]
	Define a \emph{germ of deformation (parametrized by $(S,0)$) of the holonomy of $\mathcal{F}$} as a morphism from $\Pi_\mathcal{F}(T)$ to $\mathrm{Homeo}^{(S,0)}(T)$.
\end{definition}

If $(\mathcal{F}^s)_{s \in S}$ is a deformation of $\mathcal{F}$ parametrized by $S$, then the family $(\mathcal{F}^s)_{s \in S}$ can be seen as a foliation $\mathcal{F}^S$ parametrized by $S$ of $M \times S$ of the same dimension as that of $\mathcal{F}$ such that each $(M \times \text{point})$ is saturated by leaves. In this case $T \times S$ is a transversal for the foliation $\mathcal{F}^S$ of $M \times S$ in a neighborhood of $T \times \{ 0 \}$. It follows from Epstein's second axiom that there is a homomorphism 
\[ \mathrm{H}^S \colon \Pi_{\mathcal{F}}(T) \rightarrow \utilde{\mathrm{Homeo}}^{(S, 0)}(T) \] that to each path $c$ tangent to $\mathcal{F}$ and with endpoints on $T$ assigns the germ of the holonomy of $\mathcal{F}^S$ along $c \times \{ 0 \}$ for the transversal $T \times S$. In other words, a germ of deformation of a foliation defines a germ of deformation of its holonomy. 

In the $C^r$-category for $r>0$, one can similarly define the germ of deformations of a $C^r$ foliation $\mathcal{F}$ and the germ of deformation of the holonomy of $\mathcal{F}$ parametrized by $(S,0)$. And similarly one can define the map $H^S$ in this category to associate to a germ of deformation of $\mathcal{F}$  a germ of deformation of the holonomy of $\mathcal{F}$.  A fundamental result of Bonatti and Haefliger (\cite{bonatti1990deformations}) states that in the $C^r$-category for $r>0$, conversely one can also associate to a germ of deformation of the holonomy of a $C^r$ foliation $\mathcal{F}$  a germ of deformation of $\mathcal{F}$, and the following correspondence holds. 

\begin{thm}[Bonatti--Haefliger]
	Let $\mathcal{F}$ be a $C^r$ foliation ($r \geq 1$) of a compact smooth manifold $M$, and $(S, 0)$ be a pointed locally compact topological space. The natural map $\mathrm{H}^S$, as above, defines a bijection between the set of equivalence classes of germs of deformations parametrized by $(S, 0)$ of the foliation $\mathcal{F}$ and the set of germs of deformations parametrized by $(S,0)$ of the holonomy map $\mathrm{H}$. 
	\label{thm:Haefliger--Bonatti-original}
\end{thm}

Bonatti and Haefliger prove the surjection by utilizing the notion of microfoliation and openness of the transversality condition to $C^r$-foliations when $r\geq 1$. Given that topological transversality is not an open condition, it would be interesting to see under what condition the $C^0$-case of Bonatti--Haefliger's theorem holds. For $C^{1,0}$ foliations, we show that Bonatti--Haefliger's map is an injection. This will be used in the proof of Theorem \ref{thm:main}. 

\begin{thm} Let $\mathcal{F}$ be a $C^{1,0}$ foliation of a compact smooth manifold $M$, and $(S, 0)$ be a pointed locally compact topological space. The natural map $\mathrm{H}^S$, as above, defines an injective map between the set of equivalence classes of germs of deformations parametrized by $(S, 0)$ of the foliation $\mathcal{F}$ and the set of germs of deformations parametrized by $(S,0)$ of the holonomy map $\mathrm{H}$.
\label{thm:Haefliger--Bonatti}
\end{thm}

\begin{proof}
	The proof of injectivity in Bonatti \cite{Bonatti1993} is given for $C^r$ foliations where $r \geq 1$, but as we will see below, it works with some modification for $C^{1,0}$ foliations as well.  
	
	Choose a complete transversal $T$ such that $T$ is an embedded submanifold with a trivial normal bundle. For example, $T$ can be a union of disjoint $k$-dimensional transverse disks, where $k$ is the codimension of $\mathcal{F}$. 
	
	\textit{Step 1}: Let $\mathcal{F}_1^S$ and $\mathcal{F}_2^S$ be two germs of deformations of $\mathcal{F}$. Denote by $\mathrm{H}_1^S$ and $\mathrm{H}_2^S$ the corresponding germs of deformations of holonomy.  We should prove that if $\mathrm{H}_1^S$ is equivalent (as a morphism between \'{e}tale groupoids) to $\mathrm{H}_2^S$, then the germs $\mathcal{F}_1^S$ and $\mathcal{F}_2^S$ of deformations of $\mathcal{F}$ are equivalent. First, we show that one can assume that $\mathrm{H}_1^S$ and $\mathrm{H}_2^S$ are equal (rather than equivalent). For this, we show that there is a deformation $\mathcal{F}_3^S$ whose germ is equivalent to that of $\mathcal{F}_2^S$ and such that it induces the same germ of deformation of holonomy as the one induced by $\mathcal{F}_1^S$. 
	
	Let $T \coprod T:= \{1\} \times T \cup \{2\} \times T$ be the disjoint union of two copies of $T$. Then $T \coprod T$ is a complete transversal for $\mathcal{F}$. For $i = 1, 2$, let $\varphi_i$ be the inclusion 
	\[ \varphi_i \colon \Pi_T(\mathcal{F}) \hookrightarrow  \Pi_{T \coprod T}(\mathcal{F}) \]
	induced by identifying $T$ with $\{i \} \times T$.
	Since $\mathrm{H}_1^S$ and $\mathrm{H}_2^S$ are equivalent morphisms (Definition \ref{def: equivalent morphisms}), the following holds:  there is a homomorphism 
	\[ \mathrm{H}_3^S \colon  	\Pi_{T \coprod T}(\mathcal{F})  \rightarrow \utilde{\mathrm{Homeo}}^{(S, 0)}(T) \]
	that makes the following diagram commutative
	\[
	\begin{tikzcd}[column sep=tiny]
		& \Pi_T(\mathcal{F}) \ar[dr, "\mathrm{H}_1^S"] \arrow[dl, hook, "\varphi_1"]
		&[1cm]
		\\
		\Pi_{T \coprod T}(\mathcal{F}) \arrow[rr, dashrightarrow, "\mathrm{H}_3^S"]
		&
		&\utilde{\mathrm{Homeo}}^{(S, 0)}(T) 
		\\
		&\Pi_T(\mathcal{F}) \ar[ur, "\mathrm{H}_2^S"']\arrow[ul, hook, "\varphi_2"]
		&
	\end{tikzcd} 
	\]
	Let $\tau \colon T \rightarrow M$ be the embedding of the transversal.  For each $x \in T$ denote by $x_{1,2}$ the constant path in $\Pi_{T \coprod T}(\mathcal{F})$ with source $(1,x)$ and target $(2,x)$. Let $x_{2,1} = (x_{1,2})^{-1}$. 
	
	For each $\gamma \in \Pi_{\mathcal{F}}(T)$ with source $x \in T$ and target $y \in T$ we have
	\[ \varphi_2(\gamma) = x_{2,1} \cdot \varphi_1(\gamma) \cdot y_{1,2}, \]
	where concatenation of paths is written from left to right. Therefore for every $s \in S$ close to $0$ we have 
	\begin{align*}
		\mathrm{H}_2^s(\gamma) & = \mathrm{H}_3^s(\varphi_2(\gamma)) = \mathrm{H}_3^s(x_{2,1} \cdot \varphi_1(\gamma) \cdot y_{1,2})    \\
		& = \mathrm{H}_3^s(y_{1,2}) \circ \mathrm{H}_3^s(\varphi_1(\gamma)) \circ \mathrm{H}_3^s(x_{2,1})   \\
		& = \mathrm{H}_3^s(y_{1,2}) \circ \mathrm{H}_1^s(\gamma) \circ \mathrm{H}_3^s(x_{2,1}) . 
	\end{align*}
Note that the map $x\to x_{12}$ gives a continuous embedding of $T$ into $\Pi_{T \coprod T}(\mathcal{F})$, and $\mathrm{H}_3^S(x_{1,2})$ is a germ of a homeomorphism in $\mathrm{Homeo}^{(S,0)}(T)$ at $(x,0) \in T \times S$ that fixes $(x,0)$. Since $\mathrm{H}_3^S$ is an {\' e}tale map between {\' e}tale groupoids, there exists an open neighborhood $U_x$ in $T\times S$ around each $(x, 0) \in T\times\{0\}$  and a local homeomorphism $\psi_{U_x}$ in $\mathrm{Homeo}^{(S,0)}(T)$ such that its germ at every $z\in (T\times\{0\})\cap U_x$ is the germ $\mathrm{H}_3^S(z_{1,2})$. 

In particular, the restriction of $\psi_{U_x}$ to $(T\times\{0\})\cap U_x$ is the identity. Given that $T$ is compact, finitely many such neighborhoods $U_x$ cover $T \times \{ 0\}$ and so we can choose a small neighborhood $U$ of $T \times 0$ in $T \times S$ and a homeomorphism $\psi$ of $U$ into an open subset of $T \times S$ such that for each $x \in T$ the germ of $\psi$ at $(x, 0)$ is equal to $\mathrm{H}_3^S(x_{1,2})$. 
	
	Then for each $\gamma \in \Pi_{\mathcal{F}}(T)$ we have 
	\[ \mathrm{H}_2^S(\gamma) = \psi \circ \mathrm{H}_1^s(\gamma) \circ \psi^{-1}. \]
	
	\textit{Claim}: There exists a neighborhood $\mathcal{V}$ of $M \times \{ 0\}$ in $M \times S$ and a homeomorphism $\Psi$ of $\mathcal{V}$ into an open subset of $M \times S$ such that the restriction of $\Psi $ to $M \times \{ 0\}$ is the identity and such that there exists a neighborhood $V \subset U$ of $T \times \{ 0\}$ in $T \times S$ where $\Psi(V) \subset T \times S$ and the restriction of $\Psi $ to $V$ coincides with $\psi$.
	
	\textit{Proof of Claim}: Here is the part that the proof differs from that of \cite{Bonatti1993}. By Edwards--Kirby theorem, the group $\mathrm{Homeo}(T)$ is locally contractible. Moreover, the point whose neighborhood is being contracted can be left fixed during the contraction. See \cite[Corollary 1.1]{edwards1971deformations} and the proceeding remark.  Let $\mathcal{U}$ be a small neighborhood of the identity $\mathrm{Id}$ in $\mathrm{Homeo}(T)$, and $G \colon \mathcal{U} \times [0,1] \rightarrow \mathcal{U}$ be a contraction such that 
	\begin{align*}
		&G(\phi , 0) = \phi, \hspace{3mm} G(\phi, 1) = \mathrm{Id} \hspace{6mm} \forall \phi \in \mathcal{U}, \\
		& G(\mathrm{Id}, t) = \mathrm{Id} \hspace{6mm} \forall t \in [0,1].  
	\end{align*}

	Since $T$ is embedded, it has a closed tubular neighborhood $\mathcal{N}$. By our assumption on $T$, we can identify $\mathcal{N}$ with the trivial bundle $T \times \mathbb{D}^k$ over $T$, where $\mathbb{D}^k$ is the unit disk in $\mathbb{R}^k$. For each $s$ close to $0$ in $S$, let $\psi_s \in \mathrm{Homeo}(T)$ be the restriction of $\psi$ to $T \times \{ s\}$. We extend $\psi_s$ to a homeomorphism $\Psi_s$ of $M$ as follows. Define the restriction of $\Psi_s$ to $M - \mathcal{N}$ to be the identity map. Denote the Euclidean norm of a point $z \in \mathbb{D}^k$ by $|z|$. For each $(x, z) \in T \times \mathbb{D}^k$ define 
	\[ \Psi_s(x,z) = (G(\psi_s, |z|)(x) , z). \]
	Then we have
	\begin{align*}
		&|z|= 1 \implies \Psi_s(x , z) = (x,z), \\ 
		&z=0 \implies \Psi_s(x,0) = (\psi_s(x), 0).
	\end{align*}
	Therefore, we can define the restriction of $\Psi$ to $M \times \{s\}$ to be equal to $\Psi_s$. This proves the \textit{Claim}.
		
	Let $\mathcal{F}_3^S = \Psi(\mathcal{F}_2^S)$. This is a deformation of $\mathcal{F}$ where $\mathcal{F}_3^0=\mathcal{F}$, since the restriction of $\Psi$ to $M \times \{ 0\}$ is the identity. The deformation of holonomy of $\mathcal{F}$ for the transversal $T$ induced by $\mathcal{F}_3^S$ is exactly $\mathrm{H}_1^S$. Then $\mathcal{F}_3^S$ is equivalent to the germ $\mathcal{F}_2^S$ of deformation of $\mathcal{F}$, with the homeomorphism $\Psi$ inducing the equivalence, and the germ of deformation of holonomy of $\mathcal{F}$ induced by $\mathcal{F}_3^S$ is equal to that of $\mathcal{F}_1^S$.  This completes the construction of $\mathcal{F}_3^s$. Therefore we can assume that 
	\[ \mathrm{H}_1^S = \mathrm{H}_2^S \colon \Pi_T(\mathcal{F}) \rightarrow \utilde{\mathrm{Homeo}}^{(S, 0)}(T).\]	
	
	\textit{Step 2}: Next we define a field of disks transverse to $\mathcal{F}$. Let $q \colon N \rightarrow M$ be the normal bundle of the foliation $\mathcal{F}$, and $\sigma \colon M \rightarrow N$ be the zero section. Consider a neighborhood $U$ of the zero section $\sigma(M)$ and a submersion $\phi \colon U \rightarrow M$ coinciding with $q$ on $M$ (that is $\phi \circ \sigma = \mathrm{id}_M$) and such that the restriction of each fiber $q^{-1}(x) \cap U$ is an embedded transversal to the foliation $\mathcal{F}$. For example, fixing a Riemannian metric on $M$, the map $\phi$ can be defined via the exponential function. By choosing the neighborhood $U$ small enough, the fibers $D_x = q^{-1}(x) \cap U$ are $k$-dimensional disks transverse to $\mathcal{F}$, where $k$ is the codimension of the foliation $\mathcal{F}$. The family $\{ D_x\}_{x\in M}$ is a \emph{field of disks transverse to $\mathcal{F}$}. We may assume that for each $x \in T$, the disk $D_x$ is included in $T$. For $s$ close to $0$, the foliations $\mathcal{F}_1^s$ and $\mathcal{F}_2^s$ are transverse to the field of disks $\{D_x\}_{x \in M}$.

	\textit{Step 3}: Choose a Riemannian metric on $M$. By compactness of $M$ and local compactness of $S$, there exists $l >0$ and a neighborhood $S_0$ of $0$ in $S$ such that for every $s \in S_0$ the following two properties hold:
	\begin{enumerate}
		\item For each $x \in M$ there exists a path $\gamma_x^s$ tangent to the foliation $\mathcal{F}_1^s$ such that $\gamma_x^s(0) \in T$ and $\gamma_x^s(1) = x$ and the length $\ell(\gamma_x^s)$ is at most $l$. 
		
		\item Each path $\gamma$ tangent to $\mathcal{F}_1^s$ and of length $\ell(\gamma) \leq 2 l$ projects to a path $\overline{\gamma}$ tangent to $\mathcal{F}_2^s$ such that $\overline{\gamma}(0)= \gamma(0)$ and $\overline{\gamma}(t) \in D_{\gamma(t)}$ for every $t \in [0,1]$. Here we use Epstein's second axiom. 
	\end{enumerate}
	
	\textit{Claim}: There exists a neighborhood $S_1 \subset S_0$ of $0$ in $S$ such that for each $s \in S_1$ the following holds: Let $\gamma_1$ and $\gamma_2$ be two paths tangent to $\mathcal{F}_1^s$ such that $\gamma_1(0)$ and $\gamma_2(0)$ lie on $T$, and $\gamma_1(1)=\gamma_2(1)$, and such that the lengths $\ell(\gamma_1)$ and $\ell(\gamma_2)$ are at most $l$. Let $\overline{\gamma}_1$ and $\overline{\gamma}_2$ be the projections of $\gamma_1$ and $\gamma_2$ to $\mathcal{F}_2^s$ along the field of disks $\{ D_x \}_{x \in M}$. Then we have $\overline{\gamma}_1(1) = \overline{\gamma}_2(1)$.
	
	\textit{Proof of Claim}: Assume to the contrary, and suppose that there exists a sequence of $s_i \in S$ converging to $0$ such that for each $i$ there are paths $\gamma_1^i$ and $\gamma_2^i$ tangent to $\mathcal{F}_1^s$ and satisfying the hypothesis of the claim and such that $\overline{\gamma}_1^i(1) \neq \overline{\gamma}_2^i(1)$. Let $\sigma_i = \gamma_1^i (\gamma_2^i)^{-1}$, where concatenation is from left to right. Then $\sigma_i$ is a path tangent to $\mathcal{F}_1^i$ and of length at most $2l$. Since $s_i \in S_0$, we can project $\sigma_i$ to a path $\overline{\sigma}_i$ tangent to $\mathcal{F}_2^i$ and satisfying $\overline{\sigma}_i(0)=\sigma_i(0)$ and $\overline{\sigma}_i(1) \in D_{\sigma_i(1)} \subset T$. We first argue that $\overline{\sigma}_i(1) \neq \sigma_i(1) $: otherwise we must have $\overline{\gamma}_1^i(1) = \overline{\gamma}_2^i(1)$ (and $\overline{\sigma}_i = \overline{\gamma}_1^i (\overline{\gamma}_2^i)^{-1}$) which is not the case by hypothesis. Hence we established that $\overline{\sigma}_i(1) \neq \sigma_i(1) $. Since $\sigma_i$ have lengths bounded by $2 l$, and have endpoints on the compact transversal $T$, and $S$ is locally compact,  after passing to a subsequence we may assume that $\sigma_i$ converge to a path $\sigma$ tangent to $\mathcal{F}_1^0 = \mathcal{F}$, with endpoints on $T$, and of length $\ell(\sigma) \leq 2l$. We show that $\mathrm{H}_1^S(\sigma) \neq \mathrm{H}_2^S(\sigma)$ to arrive at a contradiction. This is because for each realizations $h_1$ and $h_2$ of the germs $\mathrm{H}_1^S$ and $\mathrm{H}_2^S$ and for $i$ large enough we have
	\[ h_1(\sigma_i(0)) = \sigma_i(1) \neq \overline{\sigma}_i(1) = h_2(\sigma_i(0)).  \]
	The contradiction $\mathrm{H}_1^S \neq \mathrm{H}_2^S$ proves the \textit{Claim}. 
	
	\textit{Step 4}: For each $x \in M$ and each $s \in S_1$, there exists a path $\gamma_x^s$ tangent to $\mathcal{F}_1^s$, starting on $T$ and terminating at $x$, and with length $\ell(\gamma_x^s) \leq l$. By the above \emph{Claim}, the endpoint of the projection $\overline{\gamma}_x^s$ of $\gamma_x^s$ to $\mathcal{F}_2^s$ does not depend on the choice of $\gamma_x^s$. Define $\psi^s(x) = \overline{\gamma}_x^s(1)$. 
	
	We show that for $s$ close to $0$, the map $\psi^s$ is a homeomorphism of $M$. Note that $\psi^s$ is a local homeomorphism: if $\psi^s(x)$ is defined via the path $\gamma_x^s$ tangent to $\mathcal{F}_1^s$, then we can use paths tangent to $\mathcal{F}_1^s$ and almost parallel to $\gamma_x^s$ to define $\psi^s(y)$ for $y$ close to $x$. Since $\psi^s$ is continuous and $M$ is compact, the image of $\psi^s$ is compact and hence a closed subset of $M$. On the other hand, the image of $\psi^s$ is open since $\psi^s$ is a local homeomorphism. Therefore, the image of $\psi^s$ is all of $M$, as $M$ is connected. Now $\psi^s$ being a local homeomorphism and also surjective, it is a covering map. It is enough to argue that the degree of $\psi^s$ is equal to one. Note that the space of continuous maps $C(M,M)$ is a Banach manifold, and so it is locally contractible. Since $\psi^0$ is the identity map, $\psi^s$ lie in a small neighborhood of $\psi^0$, and the degree is invariant under homotopy, it follows that the degree of $\psi^s$ is equal to one. This shows that for $s$ close to $0$, the map $\psi^s$ is a homeomorphism of $M$. Moreover, $\psi^s$ continuously varies with $s$ (by Epstein's second axiom), the map $\psi^0$ is the identity, and $\psi^s(\mathcal{F}_1^s) = \mathcal{F}_2^s$. Therefore, the germs of $\mathcal{F}_1^S$ and $\mathcal{F}_2^S$ are equivalent. 	
\end{proof}

\section{Proof of Theorem \ref{thm:main}}

\begin{proof}[Proof of Theorem \ref{thm:main}]

Let $\mathcal{F}$ be a codimension-one transversely oriented $C^{1,0}$ foliation of a compact manifold $M$, and let $L$ be a compact leaf of $\mathcal{F}$ such that $\pi_1(L)$ is amenable and $H^1(L; \mathbb{R})=0$. Here we have suppressed the base point from $\pi_1(L)$. By Witte-Morris and Navas, we know that $L$ is globally stable and so $(M, \mathcal{F})$ is either the product foliation $L \times [0,1]$ or $(M, \mathcal{F})$ is the foliation induced by fibers of a fibration over $S^1$ with fiber $L$. Let $T$ be a complete transversal for $\mathcal{F}$ which is either
\begin{enumerate}
	\item[i)] $\text{point} \times [0,1]$ if $(M, \mathcal{F})$ is the product foliation $L \times [0,1]$; or
	\item[ii)] a circle that intersects every leaf once if $(M, \mathcal{F})$ is a fibration over $S^1$ with fiber $L$.
\end{enumerate}
We show that if $\mathcal{F}_n$ is a sequence of transversely oriented codimension-one $C^{1,0}$ foliations that $C^0$-approximate $\mathcal{F}$, then $\mathcal{F}_n$ is topologically equivalent to $\mathcal{F}$ for large $n$. 

Let $S$  be the topological space $\{ 0\} \cup \{ \frac{1}{n} | n \in \mathbb{N}\} \subset \mathbb{R}$ with the subspace topology. Let $\mathcal{F}^S$ be the foliation of $M \times S$ such that the restriction of $\mathcal{F}^S$ to $M \times \{ \frac{1}{n}\}$ is $\mathcal{F}_n$, and the restriction of $\mathcal{F}^S$ to $M \times \{ 0\}$ is $\mathcal{F}$. Then $\mathcal{F}^S$ defines a germ of deformation of $\mathcal{F}$. Denote by $\Pi_T({\mathcal{F}})$ the fundamental groupoid of $\mathcal{F}$ for the transversal $T$. Let $\mathrm{Homeo}_+^{(S,0)}(T)$ be the pseudogroup of local homeomorphisms of $T \times S$ of the form $(h_s(x), s)$ where $h_s$ is a local orientation-preserving homeomorphism of $T$ continuously varying with $s$. Denote by $\utilde{\mathrm{Homeo}}_+^{(S,0)}(T)$ the \'{e}tale  groupoid of germs of elements of $\mathrm{Homeo}_+^{(S, 0)}(T)$ at points of $T \times \{ 0\}$. Let 
\[ \mathrm{H}^S \colon \Pi_T(\mathcal{F}) \rightarrow \utilde{\mathrm{Homeo}}_+^{(S,0)}(T) \]
be the germ of deformation of the holonomy induced by the germ of deformation $\mathcal{F}^S$ of the foliation $\mathcal{F}$. By Theorem \ref{thm:Haefliger--Bonatti}, it is enough to show that $H^S$ is equivalent (as a morphism of \'{e}tale groupoids) to the germ of deformation of holonomy induced by the germ of trivial (i.e. constant) deformation $\mathcal{F}^S_c$ of $\mathcal{F}$
\[ \mathrm{H}^S_c  \colon \Pi_T(\mathcal{F}) \rightarrow \utilde{\mathrm{Homeo}}_+^{(S,0)}(T),  \]
where the restriction of $\mathcal{F}^S_c$ to each $M \times \{ s\}$ is $\mathcal{F}$.

Note that by i)--ii), for every $t \in T$ there is a copy of the fundamental group $\pi_1(L)$ in $\Pi_T(\mathcal{F})$ by restricting to those paths that start and end on $t$ and remain inside a leaf; denote this copy of $\pi_1(L)$ by $\pi_1(L \times t)$. Since $\mathrm{H}^S$ is a homomorphism of \'{e}tale groupoids, and $\pi_1(L \times t)$ is a group, the image of $\pi_1(L \times t)$ under $\mathrm{H}^S$ is a group. Moreover this image $\mathrm{H}^S(\pi_1(L \times t))$ is an amenable group since it is a quotient of an amenable group. On the other hand, the following diagram is commutative
\begin{eqnarray*}
	\begin{tikzcd}
		&  \utilde{\mathrm{Homeo}}_+^{(S, 0)}(T) \arrow[d, "\pi"] \\
		\Pi_{\mathcal{F}}(T) \arrow[ru, "\mathrm{H}^S"] \arrow[r, "\mathrm{H}"] & \utilde{\mathrm{Homeo}}_+(T)
	\end{tikzcd}
\end{eqnarray*}
where $\mathrm{H}$ is the holonomy homomorphism for $\mathcal{F}$, and $\pi$ is the natural projection from $\utilde{\mathrm{Homeo}}_+^{(S, 0)}(T)$ to $\utilde{\mathrm{Homeo}}_+(T)$. Therefore, the image $\mathrm{H}^S(\pi_1(L \times t))$ lies in the subgroup $G_t$ of $\utilde{\mathrm{Homeo}}_+^{(S, 0)}(T)$ consisting of germs of homeomorphisms of $T \times S$ at the point $t \times 0$. Note that $G_t$ is isomorphic to the group $\utilde{\mathrm{Homeo}}_+^{(S,0)}(\mathbb{R}, 0)$, and hence $G_t$ is left orderable. It follows that $\mathrm{H}^S(\pi_1(L \times t))$ is a left orderable group, and by Witte-Morris' theorem, it should be the trivial group. This implies that $\mathrm{H}^S(\pi_1(L \times t))$ consists only of the germ of the identity homeomorphism of $T \times S$ at the point $(t, 0)$. Since this holds for every $t \in T$, it follows that the homomorphism $\mathrm{H}^S$ is equal to (and hence also equivalent to) $\mathrm{H}^S_c$. This completes the proof. 
\end{proof}

\section{Questions}

A countable group $G$ is left orderable if and only if it is a subgroup of $\mathrm{Homeo}_+(\mathbb{R})$. Therefore, Navas' theorem shows that every countable subgroup of $\utilde{\mathrm{Homeo}}_+(\mathbb{R}, 0)$ is also a subgroup of $\mathrm{Homeo}_+(\mathbb{R})$. Andres Navas has asked the following:

\begin{question}[Navas]
	Is every countable subgroup of $\utilde{\mathrm{Homeo}_+}(\mathbb{R}^2, 0)$ also a subgroup of $\mathrm{Homeo}_+(\mathbb{R}^2, 0)$?
\end{question}

More generally one can ask the following:

\begin{question}
	Do any of the equivalences in Corollary \ref{thm:nontrivial homomorphism} hold for $\mathbb{R}^n$ for $n>1$?
\end{question}

\bibliographystyle{alpha}
\bibliography{foliation-reference}
	
\end{document}